\documentclass[10pt]{article}
\font\smallit=cmti10

\usepackage{amssymb,latexsym,amsmath,epsfig,amsthm} 
\usepackage{graphicx}
\usepackage{url}
\usepackage{multirow}
\usepackage{pdflscape}

\makeatletter

\renewcommand\section{\@startsection {section}{1}{\z@}
{-30pt \@plus -1ex \@minus -.2ex}
{2.3ex \@plus.2ex}
{\normalfont\normalsize\bfseries\boldmath}}

\renewcommand\subsection{\@startsection{subsection}{2}{\z@}
{-3.25ex\@plus -1ex \@minus -.2ex}
{1.5ex \@plus .2ex}
{\normalfont\normalsize\bfseries\boldmath}}

\renewcommand{\@seccntformat}[1]{\csname the#1\endcsname. }
\DeclareMathOperator{\lcm}{lcm}

\makeatother

\newtheorem{theorem}{Theorem}

\newtheorem{corollary}{Corollary}

\theoremstyle{definition}
\newtheorem{definition}{Definition}

\newtheorem{example}{Example}

\begin{document}

\begin{center}
\uppercase{\bf Tabulating Absolute Lucas Pseudoprimes}
\vskip 20pt
{\bf Chloe Helmreich}\\
{\smallit Department of Mathematical Sciences, Butler University, Indianapolis, IN, USA}\\
{\tt chelmreich@butler.edu}\\ 
\vskip 10pt
{\bf Jonathan Webster}\\
{\smallit Department of Mathematical Sciences, Butler University, Indianapolis, IN, USA}\\
{\tt jewebste@butler.edu}\\ 
\end{center}


\centerline{\bf Abstract}
\noindent
In 1977, Hugh Williams studied Lucas pseudoprimes to all Lucas sequences of a fixed discriminant.  These are composite numbers analogous to Carmichael numbers and they satisfy a Korselt-like criterion:  $n$ must be a product of distinct primes and $p_i - \delta_{p_i} | n - \delta_n $ where $\delta_n$ is a Legendre symbol with the first argument being the discriminant of the Lucas sequence.   Motivated by tabulation algorithms for Carmichael numbers, we give algorithms to tabulate these numbers and provide some asymptotic analysis of the algorithms.  We show that there are only finitely many absolute Lucas pseudoprimes $n = \prod_{i = 1}^k p_i$ with a given set of $k-2$ prime factors.  We also provide the first known tabulation up to $2^{64}$ for discriminant $5$.

\pagestyle{myheadings}
\thispagestyle{empty}
\baselineskip=12.875pt
\vskip 30pt

\section{Introduction}

 A base $a$ Fermat pseudoprime is a composite integer $n$ such that
\[ a^{n-1} - 1 \equiv 0 \pmod{n}. \]
It is well known that Carmichael numbers are the composite integers for which that congruence holds for all $a$ such that $(a,n) = 1$.  Korselt showed that such a number $n$ is a product of $k > 2$ distinct primes $p_1, p_2, \ldots, p_k$ and $p_i -1 | n - 1$.  The least example is $561 = 3 \cdot 11 \cdot 17$.  From a computational view, Fermat's Little Theorem was a step into primality testing and  Carmichael numbers are a roadblock to this being a successful test.  There are two notable approaches to overcoming this obstacle.  The first is by strengthening the Fermat test by considering the factors arising from a difference of squares factorization of $a^{n-1} - 1$ (e.g. \cite{rabin}).  A second approach combines a Fermat test with a seemingly conflicting test based on Lucas sequences. An example of this would be the Baillie-PSW test \cite{bpsw, bpsw-strong}, which is what GMP currently implements \cite{gmp}.  Another example would be Gratham's Frobenius pseudoprimes \cite{frob}.  The pseudoprimes to the Lucas sequences are our motivating interest.   

Since Carmichael numbers inform us about the reliability of the Fermat test, it would make sense to examine the analogous numbers for Lucas sequences.  These numbers are, perhaps, less well-known.   H.C. Williams showed that these numbers also satisfy a Korselt-like criterion \cite{williams}.  Using this result as a starting point, we continue a study of these numbers from an algorithmic point of view with an aim of tabulating them.   Our key contributions are as follows:

\begin{enumerate}
\item We prove theorems establishing finiteness and boundedness conditions.  The versions of these theorems for Carmichael numbers were initially proved by Beeger for a prime $P$ and generalized by Duparc for $P$ being composite \cite{beeger, duparc}.  
\item We provide an algorithmic interpretation of these theorems in the spirit of \cite{pinch_carmichael, sw_carm}.  In particular, the bounds on two primes are $O(P^2)$ and $O(P^3)$ but we can find both primes after creating only $O(P (\log P)^2)$ candidates.  
\item We implemented the algorithms in C\texttt{++} and tabulated all absolute Lucas pseudoprimes less than $2^{64}$ using discriminant $5$ .
\end{enumerate}

Since the first Lucas sequence used in the Baillie-PSW test is that of the Fibonacci sequence\footnote{Technically, there are at least $8$ different ways to choose the specific parameters for the Lucas sequence.  Method $A$, $A^{\star}$, $B$, $B^{\star}$ all use a Lucas sequence with $d=5$ as the first check.  Method $A$, which GMP implements, uses the parameters $(1,-1)$, e.g. the Fibonacci sequence.}, our computations will deal with Lucas sequences having discriminant $5$.  However, our results apply to any family (by discriminant) of Lucas sequences.  

The rest of the paper is organized as follows.  Section 2 gives the background on Lucas sequences, defines what absolute Lucas pseudoprimes are, and concludes with the Korselt-like criterion.  Section 3 is a comment on how we will account for asymptotic cost.  Section 4 establishes the new theorems providing bounds that may be used for algorithmic purposes.  Sections 5 and 6 state algorithms for tabulating these numbers and provide some asymptotic analysis; these two sections are bifurcated by a ``small" input size vs a ``large" input size.  Finally, section 7 addresses the practical issues with the implementation and provides some statistics on the tabulation.

\section{Lucas Sequences}

There are many equivalent definitions of the Lucas $U$-sequence.  We state two of them and encourage the reader to consult standard sources (such as \cite{leh, williams2}) for a more robust account.  First, they may be defined by expressions involving roots of a certain polynomial:  
\[U_n = U_n(A, B) = (\alpha^n-\beta^n)/(\alpha-\beta),\]
where $\alpha, \beta$ are the zeros of $x^2-Ax+B$, and $A$, $B$ are relatively prime integers with $A > 0$.  Let the discriminant be $d = A^2-4B$.   Alternatively, they we may define these sequences with a recurrence relation: 
\[U_0(A, B)=0, U_1(A, B)=1, \mbox{ and } U_n(A, B) = AU_{n-1}(A, B)-BU_{n-2}(A, B).\]
This latter definition is used to derive identities that allow efficient computation of $U_n(A,B) \pmod{m}$ for large $n$ with an algorithm akin to square-and-multiply\cite{jq96}.    We will frequently suppress the $A,B$ notation and will be implicitly working with a given family of Lucas sequences all with the same discriminant. 
 
\begin{theorem} [Analog of Fermat's Little Theorem]
    If $p$ is an odd prime and $p \nmid dB$, then \[U_{p-\left(\frac{d}{p} \right)}(A, B) \equiv 0 \pmod p.\] 
\end{theorem}

As with Fermat Little theorem, the contrapositive of this theorem can be used to detect if an integer is composite.  And, one can find composite numbers for which the contrapostive of the above theorem does not detect, which motivates the following definition.  

\begin{definition} 
    An \textit{ $(A,B)$-Lucas pseudoprime} is a composite integer $n$ satisfying \[U_{ n - \delta_n }(A,B) \equiv 0 \pmod n \] 
    where $\delta_n$ is the Jacobi symbol $\left(\frac{d}{n} \right)$.
\end{definition}

For example, the Fibonacci pseudoprimes (A081264) are $(1,-1)$-Lucas pseudoprimes.  The first $15$ are:  $323$, $377$, $1891$, $3827$, $4181$, $5777$, $6601$, $6721$, $8149$, $10877$, $11663$, $13201$, $13981$, $15251$, and $17119$.

\begin{definition} 
    An \textit{absolute Lucas pseudoprime (to the discriminant $d$)} is a composite integer $n$ satisfying \[U_{ n - \delta_n }(A,B) \equiv 0 \pmod n \] for all $A,B$ with $d = A^2 - 4B$ and $(n , dB) = 1$, where $\delta_n$ is the Jacobi symbol $\left(\frac{d}{n} \right)$.
\end{definition}

The numbers $323, 6601, 6721, 11663,$ and $17119$ are absolute Lucas pseudoprimes from the above $15$ Fibonacci pseudoprimes.  This can be checked with a Korselt-like criterion. 
  
\begin{theorem} [Williams' Criterion \cite{williams}] \label{william_crit}
    A composite number $n$ is an absolute Lucas pseudoprime if and only if $n$ is squarefree and $p - \delta_p | n - \delta_n$ for all prime divisors $p$ of $n$.
\end{theorem}

If $d=1$, the absolute Lucas pseudoprimes are Carmichael numbers and the divisibility statement in William's Korselt-like criterion becomes $(p-1)|(n-1)$.  In the algorithms for tabulating Carmichael numbers, it was common to need the Carmichael function $\lambda(n)$.  We will need a similar function but only state what values it takes for square-free numbers, which is our only concern.   

\begin{definition}
Let $n$ be a product of distinct primes.  That is, $n = \prod_{i=1}^r p_i$.  Then define $\lambda_d(n) = \lcm( p_1 - \delta_{p_1}, \ldots, p_r - \delta_{p_r})$.  
\end{definition}

If $d=1$, $\lambda_1(n)$ is the Carmichael function.  While the asymptotic behavior of $\lambda_1(n)$ has been well-studied (e.g. \cite{erdos2, fried}), we know of no results on $\lambda_d(n)$ for $d \not = 1$.

\section{Boundedness Theorems}

To tabulate Carmichael numbers with $k$ prime factors, the general strategy is to start with a composite number, called a \textit{preproduct}, with  $k-1$ (the ``large" case) or $k-2$ (the ``small" case) prime factors and find the remaining one or two prime factors which we will usually call $q$ and $r$.  This strategy is enabled by theorems that limit both the number and size of primes that may complete the preproduct.  We show that the Korselt-like criterion may be used to get analogous boundedness and finiteness results.

\begin{theorem}[See Proposition 1 of \cite{pinch_carmichael}] \label{boundstheorem}
Let $n$ be an absolute Lucas pseudoprime less than $B$ with $k > 2$ prime factors.
\begin{enumerate}
\item Let $r< k$ and put $P_r = \prod_{i=1}^r p_i$.  Then $p_{r+1} < (B/P_r)^{1/(k-r)}$ and $p_{r+1} - \delta_{p_{r+1}}$ is relatively prime to $p_i$ for all $i \leq r$.
\item Let $P_{k-1} =   \prod_{i=1}^{k-1} p_i$.   Then $P_{k-1}p_k \equiv \delta_{P_{k-1}}\delta_{p_k} \pmod{ \lambda_d(P_{k-1})}$ and $p_k - \delta_{p_k}$ divides $P_{k-1} - \delta_{P_{k-1}}$.
\item Each $p_i$ satisfies $p_i < \sqrt{n} < \sqrt{B}$.
\end{enumerate}
\end{theorem}

\begin{proof}
These follow from $p_i - \delta_{p_i} | n - \delta_n$.
\end{proof}

The theorem requires $k>2$;  we will address the case of $k=2$ below.  The requirement that $p_{r+1} - \delta_{p_{r+1}}$ is relatively prime to $p_i$ for all $i \leq r$ is stronger than the square-free requirement in the Korselt-like criterion.  We call a square-free composite number $P_r$ \textit{admissible} if all the primes satisfy the divisibility requirement of Theorem \ref{boundstheorem}.1.  Further, we say $P_r$ is \textit{bounds admissible} (with respect to $B$) if it also satisfies the inequality in Theorem \ref{boundstheorem}.1.  For example, let  every prime number is admissible but only primes less than $B^{1/3}$ are bounds admissible.

When $d=1$, the admissible numbers are also called cyclic (in the group theory sense) numbers.  In \cite{erdos}, Erd\H{o}s proved that the counting function of cyclic numbers is asymptotic to 
\[  \frac{e^{-\gamma}B}{\log \log \log B}, \]
where $\gamma \approx 0.5772 \ldots$ is the Euler-Mascheroni constant.  We believe that his proof holds for $d \not = 1$ due to a formal replacement of various ``$1$'s"  in the proof to some appropriate Jacobi symbol.  However, this $\log \log \log B$ plays no significant role in the analysis that follows, so we do not attempt to prove this result.  

\begin{theorem}[See Proposition 2 of \cite{pinch_carmichael}]\label{algtheorem2}
Let $n$ be an absolute Lucas pseudoprime of the form $n = Pqr$ with $q$ and $r$ primes, $q < r$, and $P > 1$.  Then, there are integers $1 \leq D  < P < C$ such that with $\Delta = CD-P^2$,
\begin{align}
q - \delta_q = \frac{(P-\delta_P)(\delta_qP + \delta_rD)}{\Delta}, \\
r - \delta_r =  \frac{(P-\delta_P) ( \delta_r P + \delta_q C)}{\Delta}, \\
 \frac{(p-1)P^2 -2P}{p + 1} < CD <  \frac{(p+3)P^2 + 2P}{p + 1} .
\end{align}
where $p$ is the largest prime dividing $P$.
\end{theorem}

\begin{proof}

Since 
\[q - \delta_q | Pqr - \delta_P\delta_q\delta_r = Pqr - Pr\delta_q + Pr\delta_q - \delta_P\delta_q\delta_r\] 
it follows that $q-\delta_q| Pr - \delta_P\delta_r$.  Similarly, $r-\delta_r | Pq - \delta_P\delta_q$.  So that we define positive integers
\[ D = \frac{ Pq - \delta_P\delta_q } { r - \delta_r} \quad \mbox{and} \quad C = \frac{Pr - \delta_P\delta_r}{q-\delta_q}.\]
satisfying $ 1 \leq D < P < C$.  We have
\[ C(q-\delta_q) = P \left( \frac{ Pq - \delta_P\delta_q } {D} + \delta_r \right) -\delta_P\delta_r \]
so that 
\[ CD(q-\delta_q) = P^2q - P\delta_P \delta_q  + PD\delta_r - D\delta_P\delta_r .\] 
Further,
\begin{align*}
 (CD-P^2)(q-\delta_q) &=  P^2\delta_q - P\delta_P \delta_q  + PD\delta_r - D\delta_P\delta_r \\
&= (P - \delta_P)(\delta_q P + \delta_r D) .
\end{align*}
Note that $\Delta = CD - P^2 \not = 0$, so that
\[ q - \delta_q = \frac{(P-\delta_P)(\delta_qP + \delta_rD)}{\Delta}. \]
and similarly
\[r - \delta_r =  \frac{ (P-\delta_P)( \delta_r P + \delta_q C)}{\Delta}. \]
Note that $p + 1 \leq q - \delta_q$ so
\[ p + 1 \leq q - \delta_q =  \frac{ (P-\delta_P)( \delta_q P + \delta_r D)}{\Delta}. \]
So,
\[ |CD - P^2| < \frac{(P+1)(P+D)}{p + 1}  < \frac{2P(P+1)}{p+1}\]
implies
\[   -\frac{2P(P+1)}{p + 1} + P^2 < CD <  \frac{(2P)(P+1)}{p + 1} +P^2\]
which is equivalent to
\[  \frac{(p-1)P^2 -2P}{p + 1} < CD <  \frac{(p+3)P^2 + 2P}{p + 1} .\]
\end{proof}

\begin{corollary}
There are only finitely many absolute Lucas pseudoprimes with $k > 2$ prime factors assuming a set of $k-2$ of the prime factors are fixed.  
\end{corollary}  

\begin{corollary}
With the notation above, $q < 2(P+1)^2$ and $r < (P+1)^3$.
\end{corollary}

An interpretation of the above corollary would imply $O(P^2 \log P)$ arithmetic operations are required to use a sieve of Eratosthenes to find candidate primes $q$ for $P$.  This, in turn, requires $\Omega( P^2 \log P)$ arithmetic operations to find $r$ because there is at least $O(1)$ arithmetic operations required for a given pair $P$ and $q$.  We will see below that we can do much better than this.    

\section{Model of Computation}

It is common to measure the asymptotic cost of an algorithm in either bit operations or arithmetic operations.  Informally, asymptotic notation (especially big-$O$) is often used as a way to give guidance about the run-time of implemented algorithms.   Our theorem statements will count the number of candidates created for $q$ or $r$ but our exposition may speak more loosely as if this were measuring time.  The theorems could be viewed as the arithmetic cost of creating $q$ and $r$ without testing if they are prime.  In which case, we could multiply these asymptotic results by the asymptotic cost of primality testing to get a result that would be an asymptotic result measuring arithmetic operations.  However, this result would not be of much guidance for the run-time of an implementation because primality testing is not often the bottle-neck.  For example, it is often the case that $q$ and $r$ may be checked with $O(1)$ arithmetic operations.  Here are some examples:  they may be too big, they may not be integers, they may not satisfy certain other divisibility statements, or they may be small enough to be in a look-up table.  So, it could be the case that the average cost is $O(1)$ arithmetic operations.  Our implementation uses strong Fermat tests with the bases $\{2, 3, 5, 7, 11\}$ and this is sufficient to prove primality for all $32$-bit integers \cite{jaeschke}.   Whenever complete or partial factorizations of $n-1$ or $n+1$ are known there are fast primality tests\footnote{It is perhaps fitting for this work that these tests are also inspired by \'Edouard Lucas and many of the variants bear his name.} (see Sections 4.1 and 4.2 of \cite{cp} or \cite{williams2} for more details).   As we will see below, it is often the case that we know a complete or partial factorization of $q - \delta_q$ or $r - \delta_r$ and so these tests would be helpful.   Given the variety of approaches that are available, we believe that it is best to provide asymptotic arguments in terms of the counts of candidates $q$ and $r$ rather than the more traditional bit or arithmetic operations.  For empirical evidence supporting this, see Example \ref{example} where about $15.6$ million candidate primes are created and the algorithm only invoked a primality test $68$ times.  

\section{Algorithms for small preproducts}

In \cite{pinch_carmichael, sw_carm}, Carmichael numbers are constructed of the form $n = Pqr$.  Here we briefly sketch what was shown before in order to show that a comparable tabulation algorithm exists for absolute Lucas pseudoprimes.  The inequality on $D$ found in Theorem \ref{algtheorem2} may be used in a for-loop.  Following the approach of \cite{pinch_carmichael}, we use the inequality on \ref{algtheorem2}.3 to construct valid $C$ for the inner for-loops.  With $C$ and $D$, one can construct $q$ and $r$, and perform the required checks.  Following the approach of \cite{sw_carm}, we use the numerator of \ref{algtheorem2}.1 and construct all possible divisors.   These divisors are efficiently obtained via the use of some variant of the sieve of Eratosthenes.   With $D$ and $\Delta$, one can construct $C$ and $r$, and perform the required checks.  Before a more thorough explanation, we deal with the smallest possible preproduct $P=1$.  This situation is unique to these numbers and cannot arise with Carmichael numbers.   

\subsection{$P = 1$}

A complete tabulation must account for the case that $n = p_1p_2$.  In \cite{williams}, it is proved that this only happens when $p_1 = p_2 - 2$, $(d|p_1) = -1$, and $(d|p_2) = 1$.  Therefore, it suffices to tabulate twin primes in set residue classes.  For example, with $d=5$ we need the primes that are $17,19 \pmod{30}$.  A straightforward implementation of the sieve of Eratosthenes finds these in $O(B^{1/2}\log \log B)$ arithmetic operations.  There are other sieving methods that can improve the time by up to a factor of $(\log \log B)^3$ \cite{pip}.  Whether a faster variant is used or not, this component of the computation contributes only to a lower-order term in the overall asymptotic cost of tabulation.  Henceforth, we assume that there are always $k > 2$ prime factors in our construction.  

\subsection{$CD$ method}

The first approach follows Pinch's method of constructing $CD$ pairs. To do so, a double nested for-loop creates $D$ satisfying $1\leq D<P$.  The inequality found in Theorem \ref{algtheorem2}.3 sets the bounds for $C$ for the second for-loop.  In the inner-loop, we check that the number implied by Theorem \ref{algtheorem2} is an absolute Lucas pseudoprime.  That is, we check that $q$ and $r$ are integral.  Second, that the divisibility statements in Theorem \ref{william_crit} hold for all primes.  Lastly, we check that both $q$ and $r$ are primes.  The ordering of those checks is not required from the point of view of the asymptotic cost of the algorithm but was chosen to delay the most expensive checks until last.

\begin{theorem}
    The number of $CD$ pairs used to tabulate all absolute Lucas pseudoprimes of the form $Pqr$ is $\Theta(P_{k-3}P \log P) \subset O(P^{2 - \frac{1}{k-2} }\log P)$. 
\end{theorem}

\begin{proof}
We start with the inequality found in the proof of Theorem \ref{algtheorem2} that bounds the length of the interval around $P^2$:
\[ |CD - P^2|  < \frac{2P(P+1)}{p+1} < 2P_{k-3}(P+1).\]
So, the interval length is bounded by $4P_{k-3}(P+1)$.  Now, the total number of $C$ values created for each $D$ is given by 
\[ \sum_{D=1}^{P-1} \left\lfloor \frac{4P_{k-3}(P+1)}{D} \right\rfloor = \Theta(P_{k-3}P \log P ).  \]
Since $P_{k-3}$ may be bounded by $P^{1 - \frac{1}{k-2} }$ (see Theorem \ref{boundstheorem}.1), this gives a bound of $O(P^{2 - \frac{1}{k-2} }\log P)$.
\end{proof}

 Due to the absolute value on the inequality above, double the work is required. For each $CD$ pair, two cases are considered.  This implies that this should be about four times slower than the $CD$ method for the Carmichael case.  Since this constant is ignored in the asymptotic analysis, the result is the same as Theorem 4 from \cite{sw_carm}.  

\subsection{$D\Delta$ method}

The second method is to construct the divisors of $(P-\delta_P)(\delta_qP+\delta_rD)$. Because of the Jacobi symbol $\delta_r$, the magnitude for $(\delta_qP+\delta_rD)$ can be any integer in $[1, 2P-1]$ (except $P$). The symbol $\delta_q$ allows these divisors to be positive or negative.   So there are a total of $4$ different cases to consider.  Our implementation, considers divisors of numbers in the interval $[1, 2P-1]$ and for each integer, we construct a set of positive divisor and a set of negative divisors (they are the same set, the algorithm just treats the two cases differently).   Thus, we implicitly account for all $4$ possible choices of Jacobi symbols.  For each of the four separate cases, we constructed $C$ by first checking it was integral. Next, we created $q$ and $r$ using the appropriate symbols. The Korselt-like criterion and primality of $r$ and $q$ were then verified. 

\begin{theorem}
    The number of $D\Delta$ pairs used to tabulate all absolute Lucas pseudoprimes of the form $Pqr$ is $O(\tau(P-\delta_P)\left( P\log P \right))$. 
\end{theorem}

\begin{proof}
    For every $P$, we run through $D$ on the interval $[1, P-1]$. Then count the number of divisors of $(P-\delta_P)(\delta_qP+\delta_rD)$. 
    \begin{eqnarray*}
        & \sum_{D<P-1}\tau \left((P-\delta_P)(\delta_qP+\delta_rD)\right) \\
        <&
        \displaystyle{  \tau(P-\delta_P)  \left(\sum_{ D <  P-1} \tau(\delta_q P  + \delta_rD) \right) }\\ 
        <& \displaystyle{  2\tau(P-\delta_P)  \left(\sum_{ n < 2P}\tau(n)   \right) }\\
        =& 2\tau(P-\delta_P)\left( 2P\log 2P + (2\gamma -1)2P 
        + O(\sqrt{2P})\right)\\
        =& 4\tau(P-\delta_P)\left( P\log P + (2\gamma +\log 2 -1)P 
        + O(\sqrt{2P})\right)\\
        =& O(\tau(P-\delta_P)\left( P\log P \right))
    \end{eqnarray*} 
    The second inequality follows from the fact that the quantity $(\delta_qP+\delta_rD)$ can be either positive or negative and ranges in values from $1$ to $2P - 1$.  The former accounts for the $2$ that appears and the latter accounts for the change in the summation. 
\end{proof}

As with the $CD$ method, this is the same asymptotic result as Theorem 5 from \cite{sw_carm} but with an implied constant that is 4 times larger.

\begin{example}\label{example}
Let $P=11\cdot13\cdot17\cdot19=46189$, then there are eight absolute Lucas pseudoprimes for $d=5$ of the form $Pqr$. 
\begin{enumerate}
    \item $P\cdot57349\cdot331111621=877079242172199781$
    \item $P\cdot709\cdot4093501=134053974841501$
    \item $P\cdot1009\cdot378901=17658567813601$
    \item $P\cdot230941\cdot29144629=310883829596647021$
    \item $P\cdot2161\cdot231589=23115923797681$
    \item $P\cdot23\cdot83=88174801$
    \item $P\cdot161659\cdot577351=4311003447437401$
    \item $P\cdot1459\cdot2251=983368161419501$
\end{enumerate}
The divisor method requires checking about 7.8 million $D\Delta$ pairs. However, the $CD$ method requires the construction of about 4.83 billion $CD$ pairs.  By prioritizing all other checks first, the $D\Delta$ method used only 68 primality checks (and 16 were required to get the above output). 
\end{example}

\section{Algorithms for large preproducts}

\subsection{Distinguishing ``large" from ``small"}

So far, the only approach to find $n < B$ has been to construct a preproduct $P = P_{k-2}$ and use Theorem \ref{algtheorem2} to find the remaining two primes in time that is essentially linear in $P$.   This approach has the benefit that it is not dependent on $k$ or $\lambda_d(P)$.   However, as $P$ grows in size (with respect to $B$) it is more and more likely to create absolute Lucas pseudoprimes outside the tabulation bound.   We may discard these but there is no obvious way to improve the asymptotic cost and only generate the $q = p_{k-1}$ and $r = p_k$ of the correct sizes.  At some point it will be more efficient to exhaustively generate the candidate $q$ values.   We have two bounds on $q$:  $Pq^2 < B$ implies $q < (B/P)^{1/2}$ and $q < 2(P + 1)^2$.  Assuming that $q$ is generated with the use of a sieve or by a look-up table of precomputed primes, the cost will be roughly linear in the length of the interval (differing by $\log B$ factors depending on the method used).  These two bounds equalize around  $P = B^{1/3}$.  For the ``large" case, we will assume that $P > X > B^{1/3}$  where $X$ is some chosen cross-over point.  We will construct $q$ by exhaustive search for primes in the interval $(p_{k-2}, \sqrt{B/P}) \subset (p_{k-2}, \sqrt{B/X}) \subset [1, B^{1/3})$.    With $q$, we now know $P_{k-1} = Pq$ and $\lambda_d(P_{k-1})$.  With this information we can analyze the cost of finding $r = p_k$.  The difficulty with getting an asymptotic estimate of the total cost of the tabulation of the ``large case" is that not much is known about the asymptotic behavior of $\lambda_d(P_{k-1})$.  For example, if $\lambda_d(P_{k-1})$ were within a fixed constant multiple $\ell$ of $P_{k-1}$, then there would only be $2\ell$ candidate values of $p_k$ to check.  However, there is no reason to believe that this could happen.  Since $\lambda_1(n)$ can be very small with respect to $n$, it would be reasonable to believe that $\lambda_d(n)$ has the same property.

\subsection{Finding $p_k$ given $P_{k-1}$}

There are many approaches for finding $p_k$ given $P_{k-1}$.  We describe what we did and discuss some valid options that were not implemented. 

Using Theorem \ref{boundstheorem}.2, we know \[p_k \equiv \delta_{P_{k-1}}\delta_{p_k} P_{k-1}^{-1} \pmod{\lambda_d(P_{k-1})}\]
which means that there are two residue classes $r_1, r_2$ modulo $\lambda_d(P_{k-1})$ to consider.  The number of candidates to be considered in this arithmetic progression is  
\[  \min\left\{ \left\lceil \frac{P_{k-1} - \delta_{P_{k-1}}}{\lambda_d(P_{k-1})} \right\rceil,\left\lceil \frac{B}{P_{k-1}\lambda_d(P_{k-1})} \right\rceil \right\}. \]
The first term comes from $(p_k - \delta_{p_k})|(P_{k-1} - \delta_{P_{k-1}})$ trivially implies  $p_k - \delta_{p_k} < P_{k-1} - \delta_{P_{k-1}}$.  The second term comes from the fact that $P_{k-1}p_k < B $ and we compute the greatest multiple of $\lambda_d(P_{k-1})$ for which the inequality holds.  This is all we implemented.

The above approach is primarily based on the fact that creating these candidates in arithmetic progression is ``fast" and memory efficient.  However, it is unlikely to be an asymptotically optimal choice.   This is because the worst-case arises when $\lambda_d(P_{k-1})$ is really small.  In which case, one should probably view the problem as integer factorization rather than one of sieving in an arithmetic progression.  That is, the real goal is to find factors of $P_{k-1} - \delta_{P_{k-1}}$.  On this view, the congruence 
\[p_k \equiv \delta_{P_{k-1}}\delta_{p_k} P_{k-1}^{-1} \pmod{\lambda_d(P_{k-1})}\]
can happen to make the factoring problem easier.  This happens whenever $\lambda_d(P_{k-1})$ is large enough (see results on \textit{divisors in residue classes} and section 4.2.3 of \cite{cp}).  When $\lambda_d(P_{k-1})$ is particularly small, then testing candidates in arithmetic progression could be worse than trial division because there would be $O(P_{k-1}/\lambda_d(P_{k-1})) = O(P_{k-1})$ candidates to check.  Trial division would only check $O(\sqrt{P_{k-1}})$ candidates and this is among the slowest of factoring algorithms.  Any asymptotically faster integer factorization algorithm will find candidates for $p_k$ in an asymptotically superior way.

\section{Implementation, Statistics, and Questions}

In section 5.3, we required divisors of integers in the interval $[1, 2P-1]$.  Since, the goal was to invoke these tabulation methods for all admissible $P < X$, one possible option was to have a precomputed factor-table of all integers less than $2X$.  This single table could be used to check the admissibility of $P$ and find the factors of $P-\delta_P$ and $\delta_qP-\delta_rD$ for any $P < X$.   However, this table would be very space intensive.  Instead, we opted for two incremental sieves and this uses only $O(\sqrt{P})$ space. If $X$ is chosen as suggested in Section 6.1, this is $O(B^{1/6})$ space.   One sieve was used to find admissible $P$ and it always stored the factors of $P-1$ and $P+1$ so that the factors of $P-\delta_P$ would be accessible.  For any admissible $P$, another incremental sieve was instantiated to factor integers in $[1, 2P-1]$  for the $\delta_qP-\delta_rD$ term.   We used MPI to have this run in parallel, and striped the work by counting admissible $P$.  

For $d=5$, we chose $X = 6\cdot 10^{6}$.  For every $P < X$, we used a hybrid approach combining both the $CD$ method and the $D\Delta$ method; that is, for a given $D$ we choose the inner-loop that would create fewer candidate values.  The program computed all possible $n= Pqr$ and we used post-processing to eliminate $n > 2^{64}$.  Our choice of $X$ means that there are no cases for $k=3$ that need to be accounted as large.  We wrote $8$ distinct programs for the large case (one for each $3 < k < 12$).   There are two obvious ways to implement these programs.  The first is to have $k$ incremental sieves.  Each sieve is instantiated to find primes starting at the prior incremental sieve's prime and going as large as allowed by the inequality in Theorem \ref{boundstheorem}.1.  While this is very space efficient, it seemed like there would be a lot of overhead.  Instead, we used a precomputed list of primes in the interval $ [1, \sqrt{B/X}) $.  If $X > B^{1/3}$, this requires $O(B^{1/3})$ storage.   For each $k > 3$, we keep track of $k-1$ pointers in the array.  At each level, we make sure that the implied product is bounds admissible.  And at the $k-2$ level, we also insure that the product exceeds $X$.

\subsection{Timing information for ``small" preproducts}
We implemented the $D\Delta$ method, the $CD$ method, and a hybrid approach.   All three programs were run on a single thread of a Xeon Phi 7210 1.30 GHz co-processor.   The timing considered two different cases to highlight the strengths of each approach.  The first case was $P < X $ and $P$ is admissible.  The second case limited $P$ to being prime.  As expected, the $CD$ method is superior on prime inputs.  For the admissible preproducts, the timing information seems to confirm that the $CD$ method does not scale as well which would be expected from Theorem 5.  

The hybrid approach appears not to offer an advantage on the prime preproducts.  The overhead of updating the sieve and divisor list is more expensive than the $CD$ method when $D$ was large (relative to $P$).  If this overhead could be avoided\footnote{If a global look-up table had been employed, then the overhead is the look-up.  We need the ability to dynamically turn off the incremental sieve.},  we believe that the hybrid method would be better (see section 3.3 of \cite{sw_carm}).  But, due to our specific implementation, it was not possible to dynamically turn off the incremental sieve.  A novel variant of the incremental sieve that increments in a backwards direction would be required.  With these two sieves running simultaneously (one for the interval $[P+1, 2P-1]$ that runs forward and one for $[1, P-1]$ that runs backwards), it would have been possible to detect when the divisor counts were consistently larger than number of $C$ values.  At which point, we could turn off the incremental sieve and finish the computation with only the $CD$ method.  


The tables below show the timing data (in seconds) for the $D\Delta$, $CD$, and hybrid methods for all pre-products up to varying bounds. \par

\[
\begin{array}{c|c|c|c}
\mbox{Admissible pre-product bound} & D\Delta& CD & \mbox{Hybrid} \\ \hline
1 \cdot 10^3 & 2 & 5 & 1 \\
2 \cdot 10^3 & 7 & 32 & 6 \\
3 \cdot 10^3 & 17 & 96 & 14 \\
4 \cdot 10^3 & 33 & 213 & 25 \\
5 \cdot 10^3  & 52 & 399 & 40 \\
6 \cdot 10^3 & 77 & 653 & 59 \\
7\cdot 10^3 & 107 & 981 & 82 \\
8\cdot 10^3 & 142 & 1433 & 109 \\
9\cdot 10^3 & 183 & 1968 & 141 \\
10\cdot 10^3 & 231 & 2646 & 176 
\end{array}
\]

\[
\begin{array}{c|c|c|c}
\mbox{Prime pre-product bound} & D\Delta& CD & \mbox{Hybrid} \\ \hline
1 \cdot 10^3 & 1 & .4 & 1 \\
2 \cdot 10^3 & 3 & 1 & 2 \\
3 \cdot 10^3 & 7 & 3 & 4 \\
4 \cdot 10^3 & 13 & 5 & 7 \\
5 \cdot 10^3  & 20 & 8 & 10 \\
6 \cdot 10^3 & 29 & 12 & 14 \\
7\cdot 10^3 & 40 & 17 & 19 \\
8\cdot 10^3 & 52 & 21 & 24 \\
9\cdot 10^3 & 67 & 27 & 31 \\
10\cdot 10^3 & 84 & 33 & 38 
\end{array}
\]
\ \\

We also implemented the small case to only run on bound admissible pre-products.  For example, let $B = 10^{15}$.  Then, it took 196 seconds to find the completions for preproducts in $[10^5 - 10^3, 10^5]$.  And it  only took 150 seconds for the interval $[10^5, 10^5 + 10^3]$. Even though the inputs on the first interval are smaller than the inputs on the second interval, the number of admissible preproducts decreases because primes are no longer bounds admissible.

\subsection{Timing information for ``large" preproducts}

In the next two tables, we see some timing data on the same machine for finding absolute Lucas pseudoprimes as described in Section 6.2.  The first table shows the timings for finding all such numbers with a fixed number of prime factors and the second table shows the timings when a cross-over of $X = B^{.35}$ is chosen.  As expected, the timing impact of having a cross-over is seen more clearly in the smaller $k$ values than the larger $k$ values.  As $k$ gets larger, it becomes very rare that a product of $k-2$ primes will be less than $X$.  Having the cross-over, as noted above, has an impact on the memory requirements if the computation assumes the existence of a look-up table.  We did not measure the impact that storage might have for these relatively small bounds (in comparison to $2^{64}$).  One would probably need to abandon a look-up table approach if a cross-over was not used.  
\begin{table}[ht]
\caption{Timing without a Crossover}
\begin{center}
\begin{tabular}{c|c|c|c|c|}
\mbox{Bound} & $k=4$ & $k=5$ & $k=6$ & $k=7$ \\ \hline
$10^{10}$ & 0.2 &  40.2 & 0.01 & - \\
$10^{11}$ & 1.1 & 0.8 & 0.2 & 0.01 \\
$10^{12}$ & 5.4 & 4.9 & 1.4 & 0.2 \\
$10^{13}$ & 27.7 &  30.5 & 12.3 & 2.3 \\
$10^{14}$ & 140.6 & 200.2 & 84.8 & 20 \\
$10^{15}$ & 664.1 & 1122.1 & 611.7 & 185 
\end{tabular}
\end{center}
\end{table}

\begin{table}[ht]
\caption{Timing with a cross-over chosen as $X=B^{.35}$}
\begin{center}
\begin{tabular}{c|c|c|c|c|}
Bound  & $k=4$ &$k=5$ & $k=6$ & $k=7$ \\ \hline
$10^{10}$ &0.1 & 0.1 & 0.02 & - \\
$10^{11}$ &0.6 & 0.7 & 0.2 & 0.01 \\
$10^{12}$ &3.2 & 4.7 & 1.4 & 0.2 \\
$10^{13}$ &17.2 & 29.5 & 11.7 & 1.9 \\
$10^{14}$ &92.5 & 173.3 & 91.6 & 20 \\
$10^{15}$ &496.7 & 964.6 & 681.8 & 184.6 
\end{tabular}
\end{center}
\end{table}

\begin{landscape}
\begin{table}
\caption{Values of $C(B)$ and $C(k, B)$}
\begin{center}
\begin{tabular}{| c | c | c | c | c | c | c | c | c | c | c | c | c | c | }  

\hline

$B$ & 2 & 3 & 4 & 5 & 6 & 7 & 8 & 9 & 10 & 11 &12 & Total & $\alpha$ \\  \hline 
$10^3$ & 1 & 0 & 0 & 0 & 0 & 0 & 0 & 0 & 0 & 0 & 0 & 1 & 0 \\  \hline  
$10^4$ & 1 & 2 & 0 & 0 & 0 & 0 & 0 & 0 & 0 & 0 & 0 & 3 & 0.1192 \\  \hline 
$10^5$ & 1 & 7 & 0 & 0 & 0 & 0 & 0 & 0 & 0 & 0 & 0 & 8 & 0.1806 \\  \hline 
$10^6$ & 9 & 22 & 3 & 0 & 0 & 0 & 0 & 0 & 0 & 0 & 0 & 34 & 0.2552 \\  \hline 
$10^7$ & 24 & 50 & 24 & 2 & 0 & 0 & 0 & 0 & 0 & 0 & 0 & 100 & 0.2857\\  \hline 
$10^8$ & 64 & 102 & 89 & 18 & 1 & 0 & 0 & 0 & 0 & 0 & 0 & 274 & 0.3047 \\  \hline 
$10^9$ & 159 & 189 & 249 & 106 & 7 & 0 & 0 & 0 & 0 & 0 & 0 & 710 & 0.3168 \\  \hline 
$10^{10}$ & 414 & 356 & 512 & 358 & 71 & 0 & 0 & 0 & 0 & 0 & 0 & 1711 & 0.3233  \\  \hline 
$10^{11}$ & 1053 & 633 & 1008 & 1040 & 316 & 17 & 0 & 0 & 0 & 0 & 0 & 4067 & 0.3281 \\  \hline 
$10^{12}$ & 2734 & 1110 & 1857 & 2703 & 1268 & 180 & 3 & 0 & 0 & 0 & 0 & 9855 & 0.3328 \\  \hline 
$10^{13}$ & 7301 & 2038 & 3344 & 6226 & 4174 & 966 & 59 & 0 & 0 & 0 & 0 & 24108 & 0.3371 \\  \hline 
$10^{14}$ & 19674 & 3737 & 5649 & 13287 & 12078 & 4288 & 490 & 6 & 0 & 0 & 0 & 59209 & 0.3409 \\  \hline 
$10^{15}$ & 53561 & 6754 & 9462 & 26821 & 31472 & 15721 & 2844 & 138 & 1 & 0 & 0 & 146774 & 0.3444 \\  \hline 
$10^{16}$ & 146953 & 12215 & 15639 & 51121 & 76397 & 50690 & 13280 & 1201 & 22 & 0 & 0 & 367518 & 0.3478  \\  \hline 
$10^{17}$ & 407779 & 22004 & 25186 & 94748 & 173721 & 148482 & 53529 & 7338 & 287 & 0 & 0 & 933074 & 0.3512 \\  \hline 
$10^{18}$ & 1142128 & 39974 & 0 & 0 & 0 & 0 & 191645 & 37528 & 2501 & 37 & 0 & 1142128 & \\  \hline 
$10^{19}$ & 3220913 & 73298 & 0 & 0 & 0 & 0 & 621182 & 165609 & 17013 & 526 & 5 & 3220913 & \\  \hline 
$2^{64}$ & 4247414 & 86228 & 0 & 0 & 0 & 0 & 839627 & 240259 & 27438 & 1004 & 10 & 4247414  &\\  \hline 

\end{tabular}
\end{center}
\end{table}
\end{landscape}

\subsection{Comparison to Carmichael numbers}

In a follow-up report on tabulating Carmichael numbers to $10^{21}$, Richard Pinch provided comparable information to our Table 3 in his Table 2 \cite{pinch_21}.  After $10^9$, the count of our numbers exceeds the counts of Carmichael numbers.  Letting $\alpha$ be as in the table above, the least order of magnitude for which $\alpha > 1/3$ is $13$ but for Carmichael numbers it is $15$.  It might be reasonable to believe that there are more of these numbers than Carmichael numbers because the presence of the product of twin primes plays a significant role in this count.  However, if one ignores this column the counts in this tabulation are always less than the Carmichael number counterpart.  We are not entirely sure why this is, but one factor is that primes dividing $d$ are not admissible.   Since the actual asymptotic behavior of Carmichael number is still subject to many open questions (e.g. \cite{contra}), we believe that the asymptotic counts of these numbers would be subject to the same problems.
\ \\
\ \\
\noindent {\bf Acknowledgements.}  The first author thanks Butler Summer Institute and the second author thanks the Holcomb Awards Committee for financial support of the project.  We both thank Anthony Gurovski for his initial contributions which included a tabulation up to $10^{17}$ for $d=5$.

\end{document}